\documentclass[12pt]{article}
\usepackage{amsmath,amssymb,amsthm}
\numberwithin{equation}{section}
\usepackage{units}
\usepackage{color}
\usepackage[T1]{fontenc}
\usepackage[utf8]{inputenc}
\usepackage{authblk}
\usepackage{bm}
\usepackage{graphicx}
\usepackage{varioref}

\usepackage{datetime}

\usepackage[colorlinks=true,
linkcolor=webgreen,
filecolor=webbrown,
citecolor=webgreen]{hyperref}

\definecolor{webgreen}{rgb}{0,.5,0}
\definecolor{webbrown}{rgb}{.6,0,0}

\usepackage[toc,page]{appendix}

\newtheorem{theorem}{Theorem}
\newtheorem*{thm}{Theorem}

\newtheorem{proposition}[theorem]{Proposition}

\title{Degree $5$ Fibonacci Sums via the Gelin-Ces\`aro Identity}

\author[]{Kunle Adegoke \\\href{mailto:adegoke00@gmail.com}{\tt adegoke00gmail.com}}

\affil{Department of Physics and Engineering Physics, \mbox{Obafemi Awolowo University}, 220005 Ile-Ife, Nigeria}

\begin{document}
\date{}

\maketitle

\begin{abstract}
\noindent Let $F_k$ be the $k$th Fibonacci number. Let $(G_k)_{k\in\mathbb Z}$ be any sequence obeying the  recurrence relation of the Fibonacci numbers. We employ the Gelin-Ces\`aro identity and an identity of Brousseau to evaluate the following sums: $\sum_{j=1}^n{(\pm 1)^{j - 1}G_j^5}$,\\ $\sum_{j = 1}^n {G_{j - 1} G_{j} G_{j + 1} G_{j + 2} G_{j + m} }$, $\sum_{j = 1}^n {(-F_{m - 3})^{n - j} ( F_{m + 2} )^j G_{j - 1} G_{j} G_{j + 1} G_{j + 2} G_{j + m} }$, and \\$\sum_{j = 1}^n {(-F_{m + 2})^{n - 2}F_{m - 3}^jG_{j + m} \left( {G_{j - 2} G_{j - 1} G_j G_{j + 1} G_{j + 2} G_{j + 3} } \right)^{ - 1} } $. Among other results, we evaluate the sum and alternating sum of products of five consectutive Fibonacci-like numbers, namely $\sum_{j = 1}^n {\left( { \pm 1} \right)^{j - 1} G_j G_{j + 1} G_{j + 2} G_{j + 3} G_{j + 4} }$.
\end{abstract}
\noindent 2010 {\it Mathematics Subject Classification}:
Primary 11B39; Secondary 11B37.

\noindent \emph{Keywords: }
Gelin-Ces\`aro identity, Fibonacci number, Lucas number, gibonacci sequence, summation identity.

\section{Introduction}

Let $F_j$ and $L_j$ be the $j$th Fibonacci and Lucas numbers. Let $(G_j(a,b))_{j\in\mathbb Z}$ be the gibonacci sequence having the same recurrence relation as the Fibonacci and Lucas sequences but starting with arbitrary initial values; that is, let
\begin{equation}
G_0  = a,\,G_1  = b;\,G_j(a,b)  = G_{j - 1} (a,b) + G_{j - 2}(a,b)\, (j \ge 2),
\end{equation}
where $a$ and $b$ are arbitrary numbers (usually integers) not both zero. Thus, $F_j=G_j(0,1)$ and $L_j=G_j(2,1)$. For brevity, we will always write $G_j$ for $G_j(a,b)$. Extension to negative indices is provided by writing the recurrence relation as $G_{-j}=G_{-(j - 2)} - G_{-(j - 1)}$; so that the gibonacci sequence is defined for all integers.

The gibonacci numbers can be accessed directly through the Binet-like formula
\begin{equation*}
G_j  = \frac{{(b - a\beta )\alpha ^j  + (a\alpha  - b)\beta^j }}{{\alpha  - \beta }},
\end{equation*}
where $\alpha=(1 + \sqrt 5)/2$, the golden ratio, and $\beta=-1/\alpha$.

Our main purpose in this paper is to evaluate $\sum_{j = 1}^n {\left( { \pm 1} \right)^{j - 1} G_{j + t} G_{j + t + 1} G_{j + t + 2} G_{j + t + 3} G_{j + t + m} } $  for any integers $m$, $n$ and $t$. We will also consider related sums.

As special cases, we will show that
\begin{equation*}
\begin{split}
&44\sum_{j = 1}^n {F_j F_{j + 1} F_{j + 2} F_{j + 3} F_{j + 4} }\\
&\qquad  =  - F_{n + 5}^5  + 7F_{n + 4}^5  + 47F_{n + 3}^5  + 31F_{n + 2}^5  - 9F_{n + 1}^5  - F_n^5  - 44F_{n + 4}  + 30,
\end{split}
\end{equation*}

\begin{equation*}
\begin{split}
&( - 1)^{n - 1} 44\sum_{j = 1}^n {( - 1)^{j - 1} F_j F_{j + 1} F_{j + 2} F_{j + 3} F_{j + 4} }\\
&\qquad =  - F_{n + 5}^5  + 9F_{n + 4}^5  + 31F_{n + 3}^5  - 47F_{n + 2}^5  + 7F_{n + 1}^5  + F_n^5  -44F_{n + 1} + ( - 1)^n 30,
\end{split}
\end{equation*}

\begin{equation*}
\begin{split}
&44\sum_{j = 1}^n {L_j L_{j + 1} L_{j + 2} L_{j + 3} L_{j + 4} }\\
&\qquad  =  - L_{n + 5}^5  + 7L_{n + 4}^5  + 47L_{n + 3}^5  + 31L_{n + 2}^5  - 9L_{n + 1}^5  - L_n^5  - 1100L_{n + 4}  - 4518
\end{split}
\end{equation*}
and
\begin{equation*}
\begin{split}
&( - 1)^{n - 1} 44\sum_{j = 1}^n {( - 1)^{j - 1} L_j L_{j + 1} L_{j + 2} L_{j + 3} L_{j + 4} }\\
&\qquad =  - L_{n + 5}^5  + 9L_{n + 4}^5  + 31L_{n + 3}^5  - 47L_{n + 2}^5  + 7L_{n + 1}^5  + L_n^5  -1100L_{n + 1} - ( - 1)^n 9474.
\end{split}
\end{equation*}
We will also derive the following identities as special cases of more general results:
\begin{align*}
\sum_{j = 1}^n {2^j F_j F_{j + 1} F_{j + 2}^2 F_{j + 3} }  &= \frac{2^{n + 1}}5 F_n F_{n + 1} F_{n + 2} F_{n + 3} F_{n + 4}, \\
\sum_{j = 1}^n {( - 1)^{j - 1} 3^j F_j F_{j + 1} F_{j + 2} F_{j + 3}^2 }  &= \frac{( - 1)^{n + 1} 3^{n + 1}}5 F_n F_{n + 1} F_{n + 2} F_{n + 3} F_{n + 4},\\
\sum_{j = 1}^n {3^{n - j} F_{5j + 1} }  &= \frac{1}{5}F_{5n + 5}  - 3^n.
\end{align*}
We will bring the study to a close by deriving some reciprocal sums for the gibonacci sequence, with the following special cases:
\begin{align*}
\sum_{j = 1}^n {\frac{{3^{j - 1} }}{{F_j F_{j + 2} F_{j + 3} F_{j + 4} F_{j + 5} }}} & =  - \frac{1}{5}\frac{{3^n }}{{F_{n + 1} F_{n + 2} F_{n + 3} F_{n + 4} F_{n + 5} }} + \frac{1}{{150}},\\
\sum_{j = 1}^n {\frac{{( - 1)^{j - 1} F_{j + 6} }}{{2^{3j} F_j F_{j + 1} F_{j + 2} F_{j + 3} F_{j + 4} F_{j + 5} }}}  &= \frac{1}{5}\frac{{( - 1)^{n + 1} }}{{2^{3n} F_{n + 1} F_{n + 2} F_{n + 3} F_{n + 4} F_{n + 5} }} + \frac{1}{{150}};
\end{align*}
with the limiting cases
\begin{align*}
\sum_{j = 1}^\infty {\frac{{3^{j - 1} }}{{F_j F_{j + 2} F_{j + 3} F_{j + 4} F_{j + 5} }}} & =\frac{1}{{150}},\\
\sum_{j = 1}^\infty {\frac{{( - 1)^{j - 1} F_{j + 6} }}{{2^{3j} F_j F_{j + 1} F_{j + 2} F_{j + 3} F_{j + 4} F_{j + 5} }}}  &=\frac{1}{{150}}.
\end{align*}
\section{Preliminary results}
The Gelin-Ces\`aro identity
\begin{equation*}
F_j^4  - F_{j - 2} F_{j - 1} F_{j + 1} F_{j + 2}  = 1
\end{equation*}
readily extends to the gibonacci sequence as
\begin{equation}\label{gelin}
G_j^4  - G_{j - 2} G_{j - 1} G_{j + 1} G_{j + 2}  = \lambda^2,
\end{equation}
where, here and throughout this paper, $\lambda=G_1^2 - G_0G_2$; since the identity (Vajda~\cite[Formula (28)]{vajda}):
\begin{equation}
G_jG_{j + 2}=G_{j + 1}^2 - (-1)^j\lambda
\end{equation}
implies
\begin{equation*}
\begin{split}
G_{j + 1} G_{j + 2}  &= (G_{j + 2}  - G_j )G_{j + 2} \\
 &= G_{j + 2}^2  - G_j G_{j + 2} \\
 &= G_{j + 2}^2  - G_{j + 1}^2  + ( - 1)^j \lambda \\
 &= G_j G_{j + 3}  + ( - 1)^j \lambda ;
\end{split}
\end{equation*}
or
\begin{equation}
G_jG_{j + 3}=G_{j + 1}G_{j + 2} - (-1)^j\lambda;
\end{equation}
so that
\begin{equation*}
\begin{split}
G_{j - 2} G_{j - 1} G_{j + 1} G_{j + 2} & = \left( {G_{j - 1} G_{j + 2} } \right)\left( {G_{j - 2} G_{j + 1} } \right)\\
& = \left( {G_j G_{j + 1}  + ( - 1)^j \lambda } \right)\left( {G_{j - 1} G_j  - ( - 1)^j \lambda } \right)\\
& = G_{j - 1} G_j^2 G_{j + 1}  - ( - 1)^j G_j^2 \lambda  - \lambda ^2\\ 
& = G_j^2 \left( {G_{j - 1} G_{j + 1}  - ( - 1)^j \lambda } \right) - \lambda ^2 \\
& = G_j^4  - \lambda ^2 .
\end{split}
\end{equation*}
Identity~\eqref{gelin} is also a special case of Horadam and Shannon~\cite[Identity~(2.5)]{horadam87} with $p=1$, $q=-1$ and $e=ab + a^2 - b^2$.

For arbitrary numbers $c$ and $d$, \eqref{gelin} implies
\begin{equation}\label{eq.w35pczi}
G_{j - 1} G_j G_{j + 1} G_{j + 2} \left( {cG_{j + 3}  - dG_{j - 2} } \right) = cG_{j + 1}^5  - dG_j^5  - \lambda ^2 \left( {cG_{j + 1}  - dG_j } \right).
\end{equation}
Setting $k=j - 2$ and $r=5$ in the identity (see Howard~\cite[Corollary 3.5]{howard})
\begin{equation}
F_{m + 2}G_{k + r} + (-1)^{r - 1}F_{m - r + 2}G_k = F_r G_{k + m + 2}
\end{equation}
gives
\begin{equation}
F_{m + 2}G_{j + 3} + F_{m - 3}G_{j - 2}=5G_{j + m};
\end{equation}
so that with $c=F_{m + 2}$ and $d=-F_{m - 3}$, \eqref{eq.w35pczi} becomes
\begin{equation}\label{main}
5G_{j - 1} G_j G_{j + 1} G_{j + 2} G_{j + m}  = F_{m + 2}G_{j + 1}^5  + F_{m - 3} G_j^5  - \lambda ^2 \left( {F_{m + 2}G_{j + 1}  + F_{m - 3} G_j } \right).
\end{equation}
We see from \eqref{main} that in order to evaluate the sum of the products of five gibonacci numbers, we  require both the sum of the gibonacci numbers, namely $\sum_{j=1}^n{(\pm1)^{j - 1}G_j}$, and the sum of the fifth powers of the gibonacci numbers; that is $\sum_{j=1}^n{(\pm1)^{j - 1}G_j^5}$.
\subsection{Sum and alternating sum of gibonacci numbers}
Since $G_{j + t}= G_{j + t + 2} - G_{j + t + 1}$, the identity
\begin{equation}\label{scope1}
\sum_{j = 1}^n {(f_{j + 1}  - f_j )}  = f_{n + 1}  - f_1,
\end{equation}
which is valid for any sequence $(f_k)$, gives
\begin{equation}\label{eq.linsum}
\sum_{j=1}^n G_{j + t}=G_{n + t + 2} - G_{t + 2},
\end{equation}
while summing $G_{j + t}=G_{j + t - 1} + G_{j + t - 2}$, using the identity
\begin{equation}\label{scope2}
\sum_{j = 1}^n {( - 1)^{j - 1} (f_{j + 1}  + f_j )}  = ( - 1)^{n + 1} f_{n + 1}  + f_1,
\end{equation}
produces
\begin{equation}\label{eq.linsum2}
\sum_{j=1}^n {(-1)^{j - 1}G_{j + t}}=(-1)^{n + 1}G_{n + t - 1} + G_{t - 1}.
\end{equation}
\subsection{Sum and alternating sum of the fifth powers of gibonacci numbers}
The identity
\begin{equation}\label{degree5}
G_{j + 3}^5  = 8G_{j + 2}^5  + 40G_{j + 1}^5  - 60G_j^5  - 40G_{j - 1}^5  + 8G_{j - 2}^5  + G_{j - 3}^5,
\end{equation}
derived by Brousseau~\cite[Equation (5)]{brousseau68} can be arranged as
\begin{equation}\label{arranged1}
\begin{split}
44G_j^5  &=  - \left( {G_{j + 3}^5  - G_{j + 2}^5 } \right) + 7\left( {G_{j + 2}^5  - G_{j + 1}^5 } \right) + 47\left( {G_{j + 1}^5  - G_j^5 } \right)\\
&\qquad\qquad + 31\left( {G_j^5  - G_{j - 1}^5 } \right) - 9\left( {G_{j - 1}^5  - G_{j - 2}^5 } \right) - \left( {G_{j - 2}^5  - G_{j - 3}^5 } \right),
\end{split}
\end{equation}
and also as
\begin{equation}\label{arranged2}
\begin{split}
44G_j^5  &=  - \left( {G_{j + 3}^5  + G_{j + 2}^5 } \right) + 9\left( {G_{j + 2}^5  + G_{j + 1}^5 } \right) + 31\left( {G_{j + 1}^5  + G_j^5 } \right)\\
&\qquad\qquad - 47\left( {G_j^5  + G_{j - 1}^5 } \right) + 7\left( {G_{j - 1}^5  + G_{j - 2}^5 } \right) + \left( {G_{j - 2}^5  + G_{j - 3}^5 } \right).
\end{split}
\end{equation}
Although Brousseau derived~\eqref{degree5} specifically for the Fibonacci sequence, the identity is valid for any sequence obeying the Fibonacci recurrence since only the recurrence relation featured in the derivation.

With the aide of the summation formulas~\eqref{scope1} and~\eqref{scope2}, \eqref{arranged1} and \eqref{arranged2} lead to the results stated in Proposition~\ref{lem.deg5sums}.
\begin{proposition}\label{lem.deg5sums}
If $n$ and $t$ are integers, then
\begin{equation}\label{eq.deg5sums1}
\begin{split}
44\sum_{j = 1}^n {G_{j + t}^5 }  &=  - \left( {G_{n + t + 3}^5  - G_{t + 3}^5 } \right) + 7\left( {G_{n + t + 2}^5  - G_{t + 2}^5 } \right) + 47\left( {G_{n + t + 1}^5  - G_{t + 1}^5 } \right)\\
&\qquad\qquad+ 31\left( {G_{n + t}^5  - G_t^5 } \right) - 9\left( {G_{n + t - 1}^5  - G_{t - 1}^5 } \right) - \left( {G_{n + t - 2}^5  - G_{t - 2}^5 } \right)
\end{split}
\end{equation}
and
\begin{equation}\label{eq.deg5sums2}
\begin{split}
44\sum_{j = 1}^n {(-1)^{j - 1}G_{j + t}^5 }  &=  - \left( {(-1)^{n + 1}G_{n + t + 3}^5  + G_{t + 3}^5 } \right) + 9\left( {(-1)^{n + 1}G_{n + t + 2}^5  + G_{t + 2}^5 } \right)\\
&\qquad\quad + 31\left( {(-1)^{n + 1}G_{n + t + 1}^5  + G_{t + 1}^5 } \right)- 47\left( {(-1)^{n + 1}G_{n + t}^5  + G_t^5 } \right)\\
 &\qquad\qquad\quad + 7\left( {(-1)^{n + 1}G_{n + t - 1}^5  + G_{t - 1}^5 } \right) + \left( {(-1)^{n + 1}G_{n + t - 2}^5  + G_{t - 2}^5 } \right).
\end{split}
\end{equation}

\end{proposition}
In particular,
\begin{align}
44\sum_{j = 1}^n {F_j^5 }  &=  - F_{n + 3}^5  + 7F_{n + 2}^5  + 47F_{n + 1}^5  + 31F_n^5  - 9F_{n - 1}^5  - F_{n - 2}^5  - 14,\\
44\sum_{j = 1}^n {L_j^5 }  &=  - L_{n + 3}^5  + 7L_{n + 2}^5  + 47L_{n + 1}^5  + 31L_n^5  - 9L_{n - 1}^5  - L_{n - 2}^5  - 1482,\\
( - 1)^{n - 1} 44\sum_{j = 1}^n {( - 1)^{j - 1} F_j^5 }  &=  - F_{n + 3}^5  + 9F_{n + 2}^5  + 31F_{n + 1}^5  - 47F_n^5  + 7F_{n - 1}^5  + F_{n - 2}^5  + ( - 1)^{n + 1} 14
\end{align}
and
\begin{equation}
( - 1)^{n - 1} 44\sum_{j = 1}^n {( - 1)^{j - 1} L_j^5 }  =  - L_{n + 3}^5  + 9L_{n + 2}^5  + 31F_{n + 1}^5  - 47L_n^5  + 7L_{n - 1}^5  + L_{n - 2}^5  + ( - 1)^n 74.
\end{equation}
\section{Main results}
\begin{thm}\label{main1}
If $n$, $t$ and $m$ are any integers, then
\begin{equation}\label{eq.main1}
\begin{split}
&220\sum_{j = 1}^n {G_{j + t - 1} G_{j + t} G_{j + t + 1} G_{j + t + 2} G_{j + t + m} }\\ 
&\qquad = \left( {F_{m + 2}  + F_{m - 3} } \right)\left( { - \left( {G_{n + t + 3}^5  - G_{t + 3}^5 } \right) + 7\left( {G_{n + t + 2}^5  - G_{t + 2}^5 } \right) + 47\left( {G_{n + t + 1}^5  - G_{t + 1}^5 } \right)} \right)\\
&\qquad\quad + \left( {F_{m + 2}  + F_{m - 3} } \right)\left( {31\left( {G_{n + t}^5  - G_t^5 } \right) - 9\left( {G_{n + t - 1}^5  - G_{t - 1}^5 } \right) - \left( {G_{n + t - 2}^5  - G_{t - 2}^5 } \right)} \right)\\
&\qquad\qquad - 44\lambda ^2 \left( {F_{m + 2}  + F_{m - 3} } \right)\left( {G_{n + t + 2}  - G_{t + 2} } \right) - 44F_{m + 2} \left( {G_{t + 1}^5  - \lambda ^2 G_{t + 1} } \right)\\
&\qquad\qquad\quad + 44F_{m + 2} \left( {G_{n + t + 1}^5  - \lambda ^2 G_{n + t + 1} } \right)
\end{split}
\end{equation}
and
\begin{equation}\label{eq.main2}
\begin{split}
&220\sum_{j = 1}^n {(-1)^{j - 1}G_{j + t - 1} G_{j + t} G_{j + t + 1} G_{j + t + 2} G_{j + t + m} }\\
&\qquad= -\left( {F_{m + 2}  - F_{m - 3} } \right)\left( - \left( {(-1)^{n + 1}G_{n + t + 3}^5  + G_{t + 3}^5 } \right) + 9\left( {(-1)^{n + 1}G_{n + t + 2}^5  + G_{t + 2}^5 } \right)\right)\\
&\qquad\quad - \left( {F_{m + 2}  - F_{m - 3} } \right)\left(31\left( {(-1)^{n + 1}G_{n + t + 1}^5  + G_{t + 1}^5 } \right)- 47\left( {(-1)^{n + 1}G_{n + t}^5  + G_t^5 } \right)\right)\\
 &\qquad\qquad\quad - \left( {F_{m + 2}  - F_{m - 3} } \right)\left(7\left( {(-1)^{n + 1}G_{n + t - 1}^5  + G_{t - 1}^5 } \right) + \left( {(-1)^{n + 1}G_{n + t - 2}^5  + G_{t - 2}^5 } \right)\right)\\
&\qquad\qquad + 44\lambda ^2 \left( {F_{m + 2}  - F_{m - 3} } \right)\left( {(-1)^{n + 1}G_{n + t - 1}  + G_{t - 1} } \right) + 44F_{m + 2} \left( {G_{t + 1}^5  - \lambda ^2 G_{t + 1} } \right)\\
&\qquad\qquad\quad + 44F_{m + 2} (-1)^{n + 1}\left( {G_{n + t + 1}^5  - \lambda ^2 G_{n + t + 1} } \right).
\end{split}
\end{equation}
\end{thm}
\begin{proof}
Summing ~\eqref{main} gives
\begin{equation}
\begin{split}
&5\sum_{j = 1}^n {G_{j + t - 1} G_{j + t} G_{j + t + 1} G_{j + t + 2} G_{j + t + m} } \\
&\qquad = \left( {F_{m + 2}  + F_{m - 3} } \right)\left( {\sum_{j = 1}^n {G_{j + t}^5 }  - \lambda ^2 \sum_{j = 1}^n {G_{j + t} } } \right)\\
&\qquad\quad - F_{m + 2} \left( {G_{t + 1}^5  - \lambda ^2 G_{t + 1} } \right) + F_{m + 2} \left( {G_{n + t + 1}^5  - \lambda ^2 G_{n + t + 1} } \right),
\end{split}
\end{equation}
from which~\eqref{eq.main1} follows upon using~\eqref{eq.linsum} and~\eqref{eq.deg5sums1}.

Similarly, mutiplying through~\eqref{main} by $(-1)^{j - 1}$ and summing, we have
\begin{equation}
\begin{split}
&5\sum_{j = 1}^n {(-1)^{j - 1}G_{j + t - 1} G_{j + t} G_{j + t + 1} G_{j + t + 2} G_{j + t + m} } \\
&\qquad = -\left( {F_{m + 2}  - F_{m - 3} } \right)\left( {\sum_{j = 1}^n {(-1)^{j - 1}G_{j + t}^5 }  - \lambda ^2 \sum_{j = 1}^n {(-1)^{j - 1}G_{j + t} } } \right)\\
&\qquad\quad + F_{m + 2} \left( {G_{t + 1}^5  - \lambda ^2 G_{t + 1} } \right) + (-1)^{n + 1}F_{m + 2} \left( {G_{n + t + 1}^5  - \lambda ^2 G_{n + t + 1} } \right),
\end{split}
\end{equation}
from which~\eqref{eq.main2} results upon inserting~\eqref{eq.linsum2} and~\eqref{eq.deg5sums2}.
\end{proof}
We list some examples from the Theorem.
\begin{equation}\label{eq.fib}
\begin{split}
&[t=2,m=-2]: \\
&44\sum_{j = 1}^n {F_j F_{j + 1} F_{j + 2} F_{j + 3} F_{j + 4} }\\
&\qquad  =  - F_{n + 5}^5  + 7F_{n + 4}^5  + 47F_{n + 3}^5  + 31F_{n + 2}^5  - 9F_{n + 1}^5  - F_n^5  - 44F_{n + 4}  + 30,
\end{split}
\end{equation}

\begin{equation}
\begin{split}
&[t=1,m=-1]: \\
&110\sum_{j = 1}^n {F_j^2 F_{j + 1} F_{j + 2} F_{j + 3} } \\
&\qquad= F_{n + 4}^5  - 7F_{n + 3}^5  - 25F_{n + 2}^5  - 31F_{n + 1}^5  + 9F_n^5  + F_{n - 1}^5  + 22L_{n + 2}  - 30,
\end{split}
\end{equation}

\begin{equation}
\begin{split}
&[t=1,m=0]: \\
&220\sum_{j = 1}^n {F_j F_{j + 1}^2 F_{j + 2} F_{j + 3} }\\
&\qquad=  - 3F_{n + 4}^5  + 21F_{n + 3}^5  + 185F_{n + 2}^5  + 93F_{n + 1}^5  - 27F_n^5  - 3F_{n - 1}^5  - 44L_{n + 4}  + 90,
\end{split}
\end{equation}

\begin{equation}
\begin{split}
&[t=1,m=1]:\\
&220\sum_{j = 1}^n {F_j F_{j + 1} F_{j + 2}^2 F_{j + 3} } \\
&\qquad =  - F_{n + 4}^5  + 7F_{n + 3}^5  + 135F_{n + 2}^5  + 31F_{n + 1}^5  - 9F_n^5  - F_{n - 1}^5  - 44L_{n + 3}  + 30,
\end{split}
\end{equation}

\begin{equation}
\begin{split}
&[t=1,m=2]:\\
&55\sum_{j = 1}^n {F_j F_{j + 1} F_{j + 2} F_{j + 3}^2 } \\
&\qquad=  - F_{n + 4}^5  + 7F_{n + 3}^5  + 80F_{n + 2}^5  + 31F_{n + 1}^5  - 9F_n^5  - F_{n - 1}^5  - 44F_{n + 3}  - 33F_{n + 2}  + 30,
\end{split}
\end{equation}

\begin{equation}
\begin{split}
&[t=2,m=-2]:\\
&( - 1)^{n - 1} 44\sum_{j = 1}^n {( - 1)^{j - 1} F_j F_{j + 1} F_{j + 2} F_{j + 3} F_{j + 4} }\\
&\qquad =  - F_{n + 5}^5  + 9F_{n + 4}^5  + 31F_{n + 3}^5  - 47F_{n + 2}^5  + 7F_{n + 1}^5  + F_n^5  -44F_{n + 1} + ( - 1)^n 30,
\end{split}
\end{equation}

\begin{equation}
\begin{split}
&[t=1,m=-1]:\\
&( - 1)^{n - 1} 55\sum_{j = 1}^n {( - 1)^{j - 1} F_j^2 F_{j + 1} F_{j + 2} F_{j + 3} } \\
&\qquad= F_{n + 4}^5  - 9F_{n + 3}^5  - 20F_{n + 2}^5  + 47F_{n + 1}^5  - 7F_n^5  - F_{n - 1}^5  + 11L_{n - 1}  + ( - 1)^n 30,
\end{split}
\end{equation}

\begin{equation}
\begin{split}
&[t=1,m=0]:\\
&( - 1)^{n - 1} 220\sum_{j = 1}^n {( - 1)^{j - 1} F_j F_{j + 1}^2 F_{j + 2} F_{j + 3} } \\
&\qquad=  - F_{n + 4}^5  + 9F_{n + 3}^5  + 75F_{n + 2}^5  - 47F_{n + 1}^5  + 7F_n^5  + F_{n - 1}^5  - 44L_{n + 1}  - ( - 1)^n 30,
\end{split}
\end{equation}

\begin{equation}
\begin{split}
&[t=1,m=1]:\\
&( - 1)^{n - 1} 220\sum_{j = 1}^n {( - 1)^{j - 1} F_j F_{j + 1} F_{j + 2}^2 F_{j + 3} } \\
&\qquad= 3F_{n + 4}^5  - 27F_{n + 3}^5  - 5F_{n + 2}^5  + 141F_{n + 1}^5  - 21F_n^5  - 3F_{n - 1}^5  - 44L_n  + ( - 1)^n 90,
\end{split}
\end{equation}

\begin{equation}
\begin{split}
&[t=1,m=2]:\\
&( - 1)^{n - 1} 110\sum_{j = 1}^n {( - 1)^{j - 1} F_j F_{j + 1} F_{j + 2} F_{j + 3}^2 } \\
&\qquad = F_{n + 4}^5  - 9F_{n + 3}^5  + 35F_{n + 2}^5  + 47F_{n + 1}^5  - 7F_n^5  - F_{n - 1}^5  - 22L_{n + 2}  + ( - 1)^n 30,
\end{split}
\end{equation}

\begin{equation}\label{eq.luc}
\begin{split}
&[t=2,m=-2]: \\
&44\sum_{j = 1}^n {L_j L_{j + 1} L_{j + 2} L_{j + 3} L_{j + 4} }\\
&\qquad  =  - L_{n + 5}^5  + 7L_{n + 4}^5  + 47L_{n + 3}^5  + 31L_{n + 2}^5  - 9L_{n + 1}^5  - L_n^5  - 1100L_{n + 4}  - 4518,
\end{split}
\end{equation}

\begin{equation}
\begin{split}
&[t=1,m=-1]: \\
&110\sum_{j = 1}^n {L_j^2 L_{j + 1} L_{j + 2} L_{j + 3} } \\
&\qquad= L_{n + 4}^5  - 7L_{n + 3}^5  - 25L_{n + 2}^5  - 31L_{n + 1}^5  + 9L_n^5  + L_{n - 1}^5  + 2750F_{n + 2}  - 6570,
\end{split}
\end{equation}

\begin{equation}
\begin{split}
&[t=2,m=-2]:\\
&( - 1)^{n - 1} 44\sum_{j = 1}^n {( - 1)^{j - 1} L_j L_{j + 1} L_{j + 2} L_{j + 3} L_{j + 4} }\\
&\qquad =  - L_{n + 5}^5  + 9L_{n + 4}^5  + 31L_{n + 3}^5  - 47L_{n + 2}^5  + 7L_{n + 1}^5  + L_n^5  -1100L_{n + 1} - ( - 1)^n 9474,
\end{split}
\end{equation}

\begin{equation}
\begin{split}
&[t=1,m=-1]:\\
&( - 1)^{n - 1} 55\sum_{j = 1}^n {( - 1)^{j - 1} L_j^2 L_{j + 1} L_{j + 2} L_{j + 3} } \\
&\qquad= L_{n + 4}^5  - 9L_{n + 3}^5  - 20L_{n + 2}^5  + 47L_{n + 1}^5  - 7L_n^5  - L_{n - 1}^5  + 1375F_{n - 1}  - ( - 1)^n 3930.
\end{split}
\end{equation}

\begin{proposition}\label{prop.lhpm2ek}
If $n$, $t$ and $m$ are any integers, then
\begin{equation}\label{eq.uiaqqqw}
\begin{split}
&5\sum_{j = 1}^n {(-F_{m - 3})^{n - j} ( F_{m + 2} )^j G_{j + t - 1} G_{j + t} G_{j + t + 1} G_{j + t + 2} G_{j + t + m} } \\
&\qquad = F_{m + 2}^{n + 1} G_{n + t + 1}^5  - ( - F_{m - 3} )^n F_{m + 2} G_{t + 1}^5  - \lambda ^2 \left( {F_{m + 2}^{n + 1} G_{n + t + 1}  - ( - F_{m - 3} )^n F_{m + 2} G_{t + 1} } \right).
\end{split}
\end{equation}
\end{proposition}
\begin{proof}
Apply the general telescoping summation identity
\begin{equation}\label{tele_general}
\sum_{j = 1}^n {d^{n - j} c^{j - 1} \left( {c\,f_{j + 1}  - d\,f_j } \right)}  = c^n f_{n + 1}  - d^n f_1
\end{equation}
to identity~\eqref{main} with $c=F_{m + 2}$ and $d=-F_{m - 3}$.
\end{proof}
On account of~\eqref{gelin}, identity~\eqref{eq.uiaqqqw} can also be written in the $\lambda$-free form
\begin{equation}\label{eq.r3aho2z}
\begin{split}
&5\sum_{j = 1}^n {(-F_{m - 3})^{n - j} ( F_{m + 2} )^j G_{j + t - 1} G_{j + t} G_{j + t + 1} G_{j + t + 2} G_{j + t + m} } \\
&\qquad = F_{m + 2}^{n + 1} G_{n + t - 1} G_{n + t} G_{n + t + 1} G_{n + t + 2} G_{n + t + 3}  - ( - 1)^n F_{m - 3}^n F_{m + 2} G_{t - 1} G_t G_{t + 1} G_{t + 2} G_{t + 3} .
\end{split}
\end{equation}
Here are some exampes from Proposition~\ref{prop.lhpm2ek}.
\begin{align}
5\sum_{j = 1}^n {2^j F_j F_{j + 1} F_{j + 2}^2 F_{j + 3} }  &= 2^{n + 1} F_n F_{n + 1} F_{n + 2} F_{n + 3} F_{n + 4},\quad [t=1,m=1], \\
5\sum_{j = 1}^n {( - 1)^{j - 1} 3^j F_j F_{j + 1} F_{j + 2} F_{j + 3}^2 }  &= ( - 1)^{n + 1} 3^{n + 1} F_n F_{n + 1} F_{n + 2} F_{n + 3} F_{n + 4},\quad [t=1,m=2],\\
5\sum_{j = 1}^n {2^j L_j L_{j + 1} L_{j + 2}^2 L_{j + 3} }  &= 2^{n + 1} L_n L_{n + 1} L_{n + 2} L_{n + 3} L_{n + 4} - 336,\quad [t=1,m=1], \\
5\sum_{j = 1}^n {( - 1)^{j - 1} 3^j L_j L_{j + 1} L_{j + 2} L_{j + 3}^2 }  &= ( - 1)^{n + 1} 3^{n + 1} L_n L_{n + 1} L_{n + 2} L_{n + 3} L_{n + 4} + 504,\quad [t=1,m=2].
\end{align}
\begin{proposition}
If $m$, $n$, $s$ and $t$ are any integers, then
\begin{equation}
5\sum_{j = 1}^n {( - F_{m - 5} )^{n - j} F_m^{j - 1} G_{5(j + t) + m + s} }  = F_m^n G_{5(n + t + 1) + s}  - ( - 1)^n F_{m - 5}^n G_{5(t + 1) + s}.
\end{equation}
\end{proposition}
\begin{proof}
Set $G_j=\alpha^j$ and $G_j=\beta^j$, in turn, in~\eqref{eq.r3aho2z} and combine using the Binet formula.
\end{proof}
We list some examples:
\begin{align}
\sum_{j = 1}^n {3^{n - j} F_{5j + 1} }  &= \frac{1}{5}F_{5n + 5}  - 3^n, \\
\sum_{j = 1}^n {3^{n - j} L_{5j + 1} }  &= \frac{1}{5}L_{5n + 5}  - \frac{{3^n 11}}{5}.
\end{align}
\begin{proposition}
If $m$, $n$, $s$ and $t$ are any integers, then
\begin{equation}
\begin{split}
&\sum_{j = 1}^n {\frac{{( - F_{m + 2} )^{n - j} F_{m - 3}^{j - 1} G_{j + t + m} }}{{G_{j + t - 2} G_{j + t - 1} G_{j + t} G_{j + t + 1} G_{j + t + 2} G_{j + t + 3} }}}\\
&\qquad= \frac{{F_{m - 3}^n }}{{5G_{n + t - 1} G_{n + t} G_{n + t + 1} G_{n + t + 2} G_{n + t + 3} }} - \frac{{( - F_{m + 2} )^n }}{{5G_{t - 1} G_t G_{t + 1} G_{t + 2} G_{t + 3} }};
\end{split}
\end{equation}
with the limiting case
\begin{equation}
\sum_{j = 1}^\infty {\frac{{(- 1)^{j - 1} F_{m - 3}^{j - 1} G_{j + t + m} }}{{F_{m + 2}^jG_{j + t - 2} G_{j + t - 1} G_{j + t} G_{j + t + 1} G_{j + t + 2} G_{j + t + 3} }}} =  \frac1{{5G_{t - 1} G_t G_{t + 1} G_{t + 2} G_{t + 3} }}.
\end{equation}
\end{proposition}
\begin{proof}
Write $j + t$ for $j$ and write~\eqref{main} in the $\lambda$-free form
\begin{equation*}
\begin{split}
&5G_{j + t - 2} G_{j + t - 1} G_{j + t} G_{j + t + 1} G_{j + t + 2} G_{j + t + m} \\
&\qquad= F_{m + 2} G_{j + t - 1} G_{j + t} G_{j + t + 1} G_{j + t + 2} G_{j + t + 3}  + F_{m - 3} G_{j + t - 2} G_{j + t - 1} G_{j + t} G_{j + t + 1} G_{j + t + 2};
\end{split}
\end{equation*}
arrange as
\begin{equation}
\begin{split}
&\frac{{G_{j + t + m} }}{{G_{j + t - 2} G_{j + t - 1} G_{j + t} G_{j + t + 1} G_{j + t + 2} G_{j + t + 3} }}\\
&\qquad= \frac{1}{5}\frac{{F_{m - 3} }}{{G_{j + t - 1} G_{j + t} G_{j + t + 1} G_{j + t + 2} G_{j + t + 3} }} + \frac{1}{5}\frac{{F_{m + 2} }}{{G_{j + t - 2} G_{j + t - 1} G_{j + t} G_{j + t + 1} G_{j + t + 2} }},
\end{split}
\end{equation}
and sum, using~\eqref{tele_general} with $c=F_{m - 3}$ and $d=-F_{m + 2}$.
\end{proof}
We give the following examples.
\begin{align}
\sum_{j = 1}^n {\frac{{3^{j - 1} }}{{F_j F_{j + 2} F_{j + 3} F_{j + 4} F_{j + 5} }}} & =  - \frac{1}{5}\frac{{3^n }}{{F_{n + 1} F_{n + 2} F_{n + 3} F_{n + 4} F_{n + 5} }} + \frac{1}{{150}},\quad [t = 2,m =  - 1],\\
\sum_{j = 1}^n {\frac{{3^{j - 1} }}{{L_j L_{j + 2} L_{j + 3} L_{j + 4} L_{j + 5} }}} & =  - \frac{1}{5}\frac{{3^n }}{{L_{n + 1} L_{n + 2} L_{n + 3} L_{n + 4} L_{n + 5} }} + \frac{1}{4620},\quad [t = 2,m =  - 1],\\
\sum_{j = 1}^n {\frac{{( - 1)^{j - 1} F_{j + 6} }}{{2^{3j} F_j F_{j + 1} F_{j + 2} F_{j + 3} F_{j + 4} F_{j + 5} }}}  &= \frac{1}{5}\frac{{( - 1)^{n + 1} }}{{2^{3n} F_{n + 1} F_{n + 2} F_{n + 3} F_{n + 4} F_{n + 5} }} + \frac{1}{{150}},\quad [t = 2,m = 4],\\
\sum_{j = 1}^n {\frac{{( - 1)^{j - 1} L_{j + 6} }}{{2^{3j} L_j L_{j + 1} L_{j + 2} L_{j + 3} L_{j + 4} L_{j + 5} }}}  &= \frac{1}{5}\frac{{( - 1)^{n + 1} }}{{2^{3n} L_{n + 1} L_{n + 2} L_{n + 3} L_{n + 4} L_{n + 5} }} + \frac{1}{{4620}},\quad [t = 2,m = 4];
\end{align}
with the limiting cases
\begin{align}
\sum_{j = 1}^\infty {\frac{{3^{j - 1} }}{{F_j F_{j + 2} F_{j + 3} F_{j + 4} F_{j + 5} }}} & =\frac{1}{{150}},\\
\sum_{j = 1}^\infty {\frac{{3^{j - 1} }}{{L_j L_{j + 2} L_{j + 3} L_{j + 4} L_{j + 5} }}} & =  \frac{1}{4620},\\
\sum_{j = 1}^\infty {\frac{{( - 1)^{j - 1} F_{j + 6} }}{{2^{3j} F_j F_{j + 1} F_{j + 2} F_{j + 3} F_{j + 4} F_{j + 5} }}}  &=\frac{1}{{150}},\\
\sum_{j = 1}^\infty {\frac{{( - 1)^{j - 1} L_{j + 6} }}{{2^{3j} L_j L_{j + 1} L_{j + 2} L_{j + 3} L_{j + 4} L_{j + 5} }}}  &= \frac{1}{{4620}}.
\end{align}
\section{Concluding comments}
It is apposite to note that the method with which we arrive at~\eqref{eq.uiaqqqw} or~\eqref{eq.r3aho2z} readily extends to any arbitrary sequence $(g_j)_{j\in\mathbb Z}$. Application of~\eqref{tele_general} to
\begin{equation*}
\begin{split}
&g_{j + 1} g_{j + 2}  \cdots g_{j + r - 1} (c\,g_{j + r}  - d\,g_j )\\
&\qquad = c\,g_{j + 1} g_{j + 2}  \cdots g_{j + r}  - d\,g_j g_{j + 1}  \cdots g_{j + r - 1}
\end{split}
\end{equation*}
with $f_j=g_j g_{j + 1}  \cdots g_{j + r - 1}$ gives the summation identity stated below in Proposition~\ref{telescope}.
\begin{proposition}\label{telescope}
If $(g_k)_{k\in\mathbb Z}$ is any sequence, $n$ and $r$ are any integers and $c$ and $d$ are any numbers, then
\begin{equation}\label{eq.m6m26qw}
\begin{split}
&\sum_{j = 1}^n {d\,^{n - j} c\,^{j - 1} g_{j + 1} g_{j + 2}  \cdots g_{j + r - 1} (c\,g_{j + r}  - d\,g_j )}\\ 
&\qquad = c\,^n g_{n + 1} g_{n + 2}  \cdots g_{n + r}  - d\,^n g_1 g_2  \cdots g_r .
\end{split}
\end{equation}
\end{proposition}
The Howard identity
\begin{equation}\label{eq.fjeqgql}
F_m G_{j + r}  - ( - 1)^r F_{m - r} G_j  = F_r G_{j + m}
\end{equation}
affords an illustration of how~\eqref{eq.m6m26qw} can be put to use. If we consider the sequence $(g_j)$ such that $g_j=G_{j + t}$ and choose $c=F_m$ and $d=(-1)^rF_{m - r}$, then we have the following result.
\begin{proposition}
If $m$, $r$, $n$ and $t$ are any integers such that $r\ge2$, then
\begin{equation}\label{eq.fj6gr0p}
\begin{split}
&F_r \sum_{j = 1}^n {( - 1)^{r(n - j)} F_{m - r}^{n - j} F_m^{j - 1} G_{j + t + 1} G_{j + t + 2}  \cdots G_{j + t + r - 1} G_{j + t + m} } \\
&\qquad = F_m^n G_{n + t + 1} G_{n + t + 2}  \cdots G_{n + t + r}  - ( - 1)^{rn} F_{m - r}^n G_{t + 1} G_{t + 2}  \cdots G_{t + r} .
\end{split}
\end{equation}
\end{proposition}
The reciprocal version of~\eqref{eq.fj6gr0p}, that is
\begin{equation}
\begin{split}
&F_r \sum_{j = 1}^n {\frac{{( - 1)^{rj} F_m^{n - j} F_{m - r}^{j - 1} G_{j + t + m} }}{{G_{j + t} G_{j + t + 1} G_{j + t + 2}  \cdots G_{j + t + r - 1} G_{j + t + r} }}} = \frac{{( - 1)^{r(n - 1) + 1} F_{m - r}^n }}{{G_{n + t + 1} G_{n + t + 2}  \cdots G_{n + t + r} }} + \frac{{( - 1)^r F_m^n }}{{G_{t + 1} G_{t + 2}  \cdots G_{t + r} }},
\end{split}
\end{equation}
which is valid for all integers $r$, $n$, $t$ and $m$ for which none of the terms in the denominator vanishes, is obtained by arranging~\eqref{eq.fjeqgql} as
\begin{equation*}
( - 1)^{r - 1} \frac{{F_{m - r} }}{{G_{j + r} }} + \frac{{F_m }}{{G_j }} = \frac{{F_r G_{j + m} }}{{G_j G_{j + r} }}
\end{equation*}
and choosing $g_j=G_{j + t}^{- 1}$,, $c=(-1)^{r - 1}F_{m - r}$ and $d=-F_m$ in~\eqref{eq.m6m26qw}.

\hrule



\hrule



\end{document}